\documentclass[preprint]{elsarticle}
\usepackage{amsbsy}
\usepackage{amsfonts}
\usepackage{amsmath}

\newdefinition{definition}{Definition}
\newtheorem{theorem}{Theorem}
\newtheorem{proposition}[theorem]{Proposition}
\newproof{proof}{Proof}

\newtheorem{lemma}{Lemma}

\begin{document}
\begin{frontmatter}
\title{The solution of a generalized secretary problem via analytic expressions}
\author[rvt]{Adam Woryna}
\ead{adam.woryna@polsl.pl}
\address[rvt]{Silesian University of Technology, Institute of Mathematics, ul. Kaszubska 23, 44-100 Gliwice, Poland}

\begin{abstract}
Given integers $1\leq k<n$, the Gusein-Zade version of a generalized secretary problem  is to choose one of the $k$ best of $n$ candidates for a secretary, which are interviewing in random order. The stopping rule in the selection is based only on the relative ranks of the successive arrivals. It is known that  the  best policy can be described by a  non--decreasing sequence $(s_1, \ldots, s_k)$ of integers with $l\leq s_l<n$ for every $1\leq l\leq k$, and conversely, any such a sequence determines the general structure of the best policy. We found a finite analytic expression for the probability of success when using the optimal policy with a sequence $(s_1, \ldots, s_k)$. We also study the problem of the construction of the optimal sequence, i.e. a  sequence which maximizes  the  corresponding probability of success. We discovered finite analytic expressions which
enable to calculate the elements $s_l$  of an optimal sequence  one by one, from $l=k$ to $l=1$. Until now, such  expressions were derived separately, and only for the values $k\leq 3$.
\end{abstract}

\end{frontmatter}

\section{Introduction and the main results}

In the paper we study the Gusein-Zade version of a generalized secretary problem (see~\cite{1}). There are many ways for presenting this optimal stopping problem. In the romantic version, instead of interviewing the candidates for a secretary, we have a bachelor  who has an occasion to meet a certain number of girls during his bachelorhood and who found out about this number (denoted further by $n$) in some miraculous way.   The bachelor wants to marry  one of the $k$ best girls, where $k$ is fixed and less than $n$. He can not be sure of success as he follows in his live the following principles: in every time he  gets to know only one of the  girls and after some time he must decide to marry her or to split up. In the latter case he starts to meet the next girl, but later on he can not go  back to any girl he decided to split up. The order in which he gets to know the girls is random, thus there are $n!$ equally likely orderings.
The bachelor is able to judge  only the present girl or the girls he met previously, and he has no idea about the attraction of the future girls. However, we assume that no two girls will turn out equally attractive for him. The problem is to find the best policy for the bachelor, i.e. the policy which maximizes the probability of the marriage to the girl that is one of the $k$ best. In the paper~\cite{1} it was proved the following general structure of such a policy:

\begin{proposition}[\cite{1}]\label{p1}
The best policy for a bachelor who wants to marry one of the $k$ best girls   is described by a certain non--decreasing sequence $(s_1, \ldots, s_k)$ of integers with $l\leq s_l<n$ for every $1\leq l\leq k$ in the following way: marry the $i_0$-th girl, where $1\leq i_0\leq n$ is the smallest integer such that there is $1\leq l\leq k$ which satisfies:
\begin{itemize}
\item[(i)] $s_l<i_0\leq s_{l+1}$ (assume $s_{k+1}:=n-1$),
\item[(ii)] the $i_0$-th girl is one of the $l$ best of $i_0$ girls met so far.
\end{itemize}
If such a number $i_0$ does not exist, then marry the $n$-th girl.
\end{proposition}

Given an arbitrary non--decreasing sequence $(s_1, \ldots, s_k)$ of integers such that $l\leq s_l<n$ for every $1\leq l\leq k$, it is natural to ask about the probability of success when using the above described policy with the sequence $(s_1, \ldots, s_k)$. Namely, we would like to know how  this probability depends on the elements of this  sequence and how to construct  a sequence  which maximizes this probability.

\begin{definition}\label{ddddef}
We call a sequence $(s_1, \ldots, s_k)$ which maximizes the probability of success for the policy described in Proposition~\ref{p1} as an {\it optimal sequence}.
\end{definition}

The classical version, i.e. the case $k=1$, was solved by Lindley (\cite{5}) by using equations arising from the principle of dynamic programming. He solved these equations by simple backward recursion and  obtained  that the only element  of an optimal sequence is equal to  the smallest integer $1\leq x\leq n-1$ such that  $H(n-1)-H(x)\leq 1$, where $H(x):=\sum_{j=1}^{x}1/j$ is the $x$-th harmonic number, as well as that  the probability of success when using  the optimal policy with an element $s_1$ is equal to $s_1(H(n-1)-H(s_1-1))/n$  (see also~\cite{13} for the survey paper). The cases $k=2,3$ were solved by using backward induction and exploiting the existence of an imbedded Markov chain. In the case $k=2$ the corresponding  analytic expressions were stated by Gilbert and Mosteller (\cite{4}) and the proof was  outlined by Dynkin and Yushkevich (\cite{14}). The case $k=3$ was derived by Quine and Law (\cite{2}).

In the present paper, we extended to an arbitrary value of $k$ the formula for an optimal sequence in the following way.

\begin{theorem}\label{mt}
Let $(s_1, \ldots, s_k)$ be a sequence such that $s_l=t_l+l-1$ for every $1\leq l\leq k$, where  each $t_l$  is defined as the smallest integer  $1\leq x\leq n-l$  satisfying  the  inequality
\begin{equation}\label{ineqqq}
d_l(x)\leq \frac{1}{l}\sum_{i=l}^{k-1}\left(\prod_{j=l+1}^i t_j\right)d_i(t_{i+1})+\frac{1}{l}\left(\prod_{j=l+1}^k t_j\right)\delta_{k,n},
\end{equation}
where
\begin{equation}\label{ajj}
\delta_{k,n}:=\left\{
\begin{array}{ll}
-\frac{1}{n}\cdot\sum_{j=1}^{n-1}\frac{1}{j},&if\;\;k=1,\\
\\
\frac{(n-k)!}{n!},&if\;\;k>1,
\end{array}
\right.
\end{equation}
and the map $d_l\colon \{1,\ldots, n-l\}\to\mathbb{R}$ is defined for each $0\leq l\leq k$ as follows:
\begin{equation}\label{aj}
d_l(x):=\left\{
\begin{array}{ll}
-1+\frac{k\cdot x}{n}+\frac{{n-x-1\choose k}}{{n\choose k}}+\frac{x}{{n\choose k}}\sum\limits_{j=1}^{x}\frac{{n-j-1\choose k-1}}{j},& if\;\;l=0,\\
\\
-\frac{k}{n}-\frac{1}{{n\choose k}}\sum\limits_{j=1}^{x}\frac{{n-j-1\choose k-1}}{j},&if\;\;l=1,\\
\\
\frac{k\cdot x!\cdot (n-l-x)!}{l\cdot (l-1)\cdot n!}\sum\limits_{j=0}^{l-2}{k-1\choose j}{n-k\choose x+l-j-1},&if\;\;1<l\leq k.
\end{array}
\right.
\end{equation}
Then $(s_1, \ldots, s_k)$ is an optimal sequence.
\end{theorem}

In the right side of (\ref{ineqqq}) we use the standard conventions for  the empty sum and the empty product and evaluate them to 0 and to 1, respectively. In particular, for $l=k$ the right side of (\ref{ineqqq}) just equals $\delta_{k,n}/k$. Hence, the above formula allows to  calculate the elements $s_l$ ($1\leq l\leq k$) one by one, from $l=k$ down to $l=1$.

In the present paper we also proved the following
\begin{theorem}\label{mtt}
The probability of success for the policy described in  Proposition~\ref{p1} is equal to
\begin{equation}\label{expr}
-\sum\limits_{i=0}^{k-1}\left(\prod_{j=1}^i t_j\right)d_i(t_{i+1})-\left(\prod_{j=1}^k t_j\right)\delta_{k,n},
\end{equation}
where $t_l:=s_l-l+1$ for $1\leq l\leq k$.
\end{theorem}

The known  constructions of the optimal sequence via analytic expressions were presented in  the similar form as in Theorem~\ref{mt} but, as we have mentioned above, only in the cases $k=1,2,3$. For  higher  values of  $k$, as well as for some other versions of this problem, the algorithms  computing the probability of success and the elements of an optimal sequence can be found in various of papers (see~\cite{15,8,7,10,11,16,3,1,9,17}). However, in contrast to the analytic solution, these methods apply mechanisms via dynamic or linear programming, and hence only numerically allow to  determine the elements of an optimal sequence.

\section{The strategy of the proofs}

Our proofs are purely elementary and only combinatorial arguments are used. At first, since no two girls are equally attractive for the bachelor, we  assign the rank  to each girl, which is an integer from 1 to $n$, i.e.  the rank 1 to the best girl, the rank 2 to the next best girl and so on. Then each of the possible $n!$ orderings of the girls defines uniquely  a permutation $\pi$ of the set $\{1,\ldots,  n\}$ such that $\pi(i)$ is the rank of the $i$-th girl for every $i\in\{1,\ldots, n\}$, and conversely, any permutation $\pi$ of the set $\{1,\ldots,  n\}$ defines in the obvious way the possible ordering of the girls.

Let   $w=(s_1, \ldots, s_k)$ be a non-decreasing sequence such that   $l\leq s_l<n$ for every $l\in\{1,\ldots, k\}$ and let us assume that the  policy from Proposition~\ref{p1} is used with the  sequence $w$. Let $\pi$ be a permutation defining the ordering of the girls. If the policy successfully chooses a candidate of top $k$, then we say that $\pi$ is a {\it lucky permutation corresponding to $w$}. We distinguish the case when there exists  $l\in\{1,\ldots, k\}$  such that for some   $i\in\{1,\ldots, n-1\}$  the  following two conditions hold:
\begin{itemize}
\item  $s_l<i\leq s_{l+1}$,
\item among the first $i$ girls, there are at most $l-1$ girls which are  more attractive than the $i$-th girl (equivalently,  the set $\{\pi(1), \ldots, \pi(i)\}$ contains at most $l$ elements which are not greater than $\pi(i)$).
\end{itemize}
We call such a number $l$ a {\it $w$-threshold} of the permutation $\pi$, and the corresponding number $i$ we call a {\it $(\pi, l, w)$-element} (see also Definition~\ref{deee1} in Section~\ref{s3}). Let now  assume  that  the permutation $\pi$ has  a $w$-threshold. If  $l_0\in\{1,\ldots, k\}$ is the smallest $w$-threshold of $\pi$,  then the bachelor using the policy will marry to the $i_0$-th girl, where $i_0$ is the smallest $(\pi,l_0,w)$-element. Otherwise (i.e. when $\pi$ has no $w$-thresholds), the bachelor will marry to the $n$-th  girl. In particular, the set $\Pi_{w}$ of all  lucky permutations corresponding to the sequence $w$ naturally splits into two subsets: the subset $\Pi_{w,1}$ of   permutations having a $w$-threshold and the subset $\Pi_{w,2}$ of  permutations without $w$-thresholds. Obviously, the probability of success when using the policy is equal to the  ratio
$$
\frac{|\Pi_{w}|}{n!}=\frac{|\Pi_{w,1}|}{n!}+\frac{|\Pi_{w,2}|}{n!}.
$$

For every $l\in\{1,\ldots, k\}$ we  define the following sets:
$$
X_l:=\{l,\ldots,n-1\},\;\;\;X^{(l)}:=X_{k-l+1}\times\ldots\times X_k,
$$
and the set   $X^{(0)}:=\{\epsilon\}$, where $\epsilon$ is the empty sequence. Further, we refer  to  the elements of  the sets $X^{(l)}$ as  {\it words} and to the elements of the sets $X_l$ as  {\it letters}.

In Theorem~\ref{tt1} (Section~\ref{s3}), we provide for every non-decreasing sequence $w=(x_1, \ldots, x_k)\in X^{(k)}$ the analytic formulae for the cardinalities of  the sets $\Pi_{w,1}$ and $\Pi_{w,2}$. To derive the formula for  $|\Pi_{w,1}|$, we consider for each $1\leq l\leq k$ and  $x_l<i\leq x_{l+1}$ the subset $S(l,i)\subseteq \Pi_{w,1}$ of all permutations $\pi$ such  that the number $l$ is the smallest $w$-threshold of $\pi$ and the number $i$ is the smallest $(\pi,l,w)$-element. In particular,  we can write
$$
|\Pi_{w,1}|=\sum\limits_{l=1}^k\sum_{i=x_{l}+1}^{x_{l+1}}|S(l, i)|.
$$
Further, for every  $\pi\in S(l,i)$, we divide the set $\{\pi(1), \ldots, \pi(i)\}$ into two subsets: the subset  of those elements which are not greater than $k$ and the subset  of those elements which are greater than $k$. Conversely, given  arbitrarily the sets $Y,Y'$ satisfying
$$
Y\subseteq \{1,\ldots, k\},\;\;\;\;Y'\subseteq \{k+1, \ldots, n\},\;\;\;|Y|+|Y'|=i,
$$
we consider the subset $S(Y,Y', l,i)\subseteq S(l,i)$ of those permutations $\pi$ for which
$$
\{1,\ldots, k\}\cap \{\pi(1), \ldots, \pi(i)\}=Y,\;\;\;\{k+1, \ldots, n\}\cap \{\pi(1), \ldots, \pi(i)\}=Y'.
$$
Then, we have:
$$
|S(l, i)|=\sum\limits_{j=1}^k\sum_{(Y, Y')\in M_j}|S(Y, Y', l, i)|,
$$
where $M_j$ ($j\in\{1,\ldots, k\}$) is the set of  those pairs $(Y,Y')$ for which  $|Y|=j$. In Proposition~\ref{pomol1}, we characterize the elements of the set $S(Y,Y', l,i)$, which allows to find the following formula for its  cardinality:
$$
|S(Y,Y', l,i)|=\min\{|Y|, l\}\cdot (i-l-1)!\cdot (n-i)!\cdot \prod_{j=1}^{l} (x_j-j+1).
$$
We use the above formula to find the cardinality of the set $S(l,i)$ and, consequently, the following formula for $|\Pi_{w,1}|$:
$$
|\Pi_{w,1}|=
n!\cdot\sum\limits_{l=1}^k\left((r_{l-1}(x_l)-r_{l-1}(x_{l+1}))\cdot \prod\limits_{j=1}^l(x_j-j+1)\right),
$$
where the map $r_l\colon \{l+1,\ldots,n\}\to\mathbb{R}$ is defined for every integer $l\leq k$ as follows:
\begin{equation}\label{rere}
r_l(x):=\left\{
\begin{array}{ll}
0,&if\;\;l<0,\\
\\
\frac{1}{x}-\frac{k}{n}-\frac{{n-x-1\choose k}}{x {n\choose k}}-\frac{1}{{n\choose k}}\sum\limits_{j=1}^x\frac{{n-j-1\choose k-1}}{j},&if\;\;l=0,\\
\\
\frac{(x-l-1)!}{x!}\left(1-\frac{1}{l{n\choose x}}\sum\limits_{j=0}^l(l-j){k\choose j}{n-k\choose x-j}\right),&if\;\;1\leq l\leq k.
\end{array}
\right.
\end{equation}
By using a similar idea as in Proposition~\ref{pomol1}, we also  characterize the elements of the set $\Pi_{w,2}$ (see Proposition~\ref{pomol3}), which gives the following  formula for $|\Pi_{w,2}|$:
$$
|\Pi_{w,2}|=k(n-k-1)!\cdot\prod\limits_{j=1}^k(x_j-j+1).
$$

In Section~\ref{s4}, we derive the formula for the elements of an optimal sequence. To this aim, we introduce the notion of an {\it optimal point} (see Definition~\ref{jjdef}) of an arbitrary map $f\colon Z\to \mathbb{R}$, where  $Z\subseteq X^{(l)}$ or $Z\subseteq X_l$ for some $l\in\{1,\ldots, k\}$. Next, we define for every $l\in\{0,1\ldots, k\}$ a map $T_l\colon X^{(k-l)}\to\mathbb{R}$ (see formula (\ref{uogt})), which constitutes a natural generalization of the map
$$
T\colon X^{(k)}\to\mathbb{R},\;\;\;T(w)=\frac{|\Pi_{w,1}|}{n!}+\frac{|\Pi_{w,2}|}{n!}.
$$
We study  the maps $T_l$ in relation to the  maps $c_l$, $D_l$ ($l\in\{0,\ldots, k\}$) and   $F_{l, w}$ ($l\in\{1,\ldots, k\}$, $w\in X^{(k-l)}$) defined as follows:
\begin{eqnarray}
c_l(x)&:=&d_l(x-l),\;\;\;\;x\in \{l+1,\ldots,n\},\label{cccd}\\
D_l(w)&:=&r_{l-1}(x_1)-T_l(w),\;\;\;\;w\in X^{(k-l)},\label{dddd}\\
F_{l,w}(x)&:=&-c_{l-1}(x)-x\cdot D_{l}(w),\;\;\;\;x\in \{l,\ldots,n-1\}\label{fffd},
\end{eqnarray}
where $x_1$ in (\ref{dddd}) denotes the first letter of a word $w\in X^{(k-l)}$ or $x_1:=n-1$ depending on whether $l<k$ or $l=k$. In particular, we obtain
$$
D_0=-T_0=-T.
$$
In Proposition~\ref{proo8}, we show how to describe the maps $D_l$  in terms of the maps $d_l$. As a result, we obtain for every $l\in\{1,\ldots, k\}$ that the right side of (\ref{ineqqq}) is equal to
$$
\frac{D_l\left(w^{(l)}\right)}{l},
$$
where  $w^{(l)}\in X^{(k-l)}$ arises from the sequence $w:=(t_1, t_2+1, \ldots, t_k+k-1)$ by deleting the first $l$ letters.
In Proposition~\ref{prop10}, for any  $l\in\{1,\ldots, k\}$ and $w\in X^{(k-l)}$, we show that if $t$ is the smallest element $x\in\{1,\ldots, n-l\}$ such that $d_l(x)\leq D_l(w)/l$, then the number $t+l-1$ is an optimal point of  the map $F_{l,w}$. Next, we show (Proposition~\ref{p8}) that an arbitrary sequence $w\in X^{(k)}$ is an optimal point of the map $T$ if and only if for every $l\in\{1,\ldots, k\}$ the $l$-th letter of $w$  is an optimal point of the map $F_{l, w^{(l)}}$. Finally, in Proposition~\ref{trud}, we show that if $w\in X^{(k)}$ is an optimal point of $T$, then $w$ must be  a non-decreasing sequence.  As a simple consequence of Theorem~\ref{tt1} and the above propositions, we obtain our main results (see Section~\ref{secpr}). The proofs  of Propositions~\ref{proo8}-\ref{trud}  are based on various combinatorial identities and on some auxiliary properties of the maps $r_l, c_l, D_l$ and $F_{l,w}$. We derive them in Section~\ref{sqa1}.

Further, we use for all $i,j\in\mathbb{Z}$ the following notations:
\begin{itemize}
\item $[i]:=\{t\in\mathbb{Z}\colon 1\leq t\leq i\}$,
\item $[i]_0:=[i]\cup\{0\}$,
\item $[i,j]:=\{t\in\mathbb{Z}\colon i<t\leq j\}$,
\item $[\leq i]:=\{t\in\mathbb{Z}\colon t\leq i\}$.
\end{itemize}

\section{The formula for the probability of success}\label{s3}

Let $Sym(n)$ be the set of all permutations of the set $[n]$. For every $\pi\in Sym (n)$ and every  $i\in[n]$ we call the image $\pi(i)$  the {\it $\pi$-rank} of the element $i$. We call the element $i$ a {\it $\pi$-candidate} if $\pi(i)\in[k]$. In particular, if  the ordering of the girls is defined by a permutation $\pi\in Sym(n)$, then the bachelor's win is to marry to the $i_0$-th girl, where $i_0\in[n]$ is an arbitrary $\pi$-candidate.

For every $\pi\in Sym(n)$ and every    $i\in[n]$ we   also consider the {\it relative $\pi$-rank} of the element $i$, i.e. the number of the elements from the set $[i]$ such that their $\pi$-ranks are not greater than $\pi(i)$; we denote this number by $\rho_\pi(i)$. In other words,  $\rho_{\pi}(i)$  is  the number of those elements from the set $\pi([i])$  which are not greater than $\pi(i)$.

\begin{definition}\label{deee1}
Let $\pi\in Sym(n)$,   $l\in[k]$ and let $w=(s_1, \ldots, s_{k})\in X^{(k)}$ be a non-decreasing sequence. We call an arbitrary element $i\in[s_{l}, s_{l+1}]$ satisfying the inequality $\rho_\pi(i)\leq l$ a $(\pi,l,w)$-{\it element}.  If the set  $[s_{l}, s_{l+1}]$ contains at least one  $(\pi,l,w)$-element, then we call the number $l$   a {\it $w$-threshold} of the permutation  $\pi$. In other words, the element $l\in[k]$ is a $w$-threshold of $\pi$ if there is $i\in[s_l, s_{l+1}]$ such that $\rho_\pi(i)\leq l$ (as before, we assume $s_{k+1}:=n-1$).
\end{definition}

Let now assume that the bachelor  uses the policy from Proposition~\ref{p1} with a sequence $w=(s_1, \ldots, s_k)$ and that the ordering of the girls is defined by a permutation $\pi\in Sym(n)$. Then $\pi$ is a lucky permutation corresponding to  $w$ if and only if one of the following conditions holds:
\begin{itemize}
\item $\pi$ has a $w$-threshold and if $l_0\in[k]$ is the smallest  $w$-threshold of $\pi$, then the smallest $(\pi,l_0,w)$-element is a $\pi$-candidate,
\item $\pi$ has no $w$-thresholds and the element $n$ is a $\pi$-candidate.
\end{itemize}

\begin{theorem}\label{tt1}
Let  $w_0=(x_1, \ldots, x_k)\in X^{(k)}$ be a non-decreasing sequence. Then the number of lucky permutations corresponding to $w_0$ and having a  $w_0$-threshold is equal to
$$
n!\cdot\sum\limits_{l=1}^k\left((r_{l-1}(x_l)-r_{l-1}(x_{l+1}))\cdot \prod\limits_{j=1}^l(x_j-j+1)\right),
$$
where the maps $r_l\colon [l,n]\to\mathbb{R}$ ($l\in[\leq k]$) are defined as in (\ref{rere}). The number of lucky permutations corresponding  to $w_0$ and having no $w_0$-thresholds is equal to
$$
k(n-k-1)!\cdot\prod\limits_{j=1}^k(x_j-j+1).
$$
\end{theorem}
\begin{proof}

Let us fix the integers $l_0,i_0$ such that $l_0\in[k]$ and  $i_0\in[x_{l_0},x_{l_0+1}]$. At first, we   determine the number of lucky permutations $\pi\in \Pi_{w_0}$ such that $l_0$ is the smallest $w_0$-threshold of $\pi$ and $i_0$ is the smallest $(\pi,l_0,w_0)$-element. Let us denote by $S(l_0, i_0)$ the set of all such permutations.

For every permutation  $\pi\in Sym(n)$, we denote
$$
\mathcal{X}_\pi:=\pi([i_0])\cap [k],\;\;\;\mathcal{X}'_\pi:=\pi([i_0])\cap [k,n].
$$
Obviously, we have
$$
\pi([i_0])=\mathcal{X}_\pi\cup \mathcal{X}'_\pi,\;\;\;\mathcal{X}_\pi\subseteq [k],\;\;\;\mathcal{X}'_\pi\subseteq [k,n],\;\;\;|\mathcal{X}_\pi|+|\mathcal{X}'_\pi|=i_0.
$$
Moreover, if  $\pi\in S(l_0, i_0)$, then $\pi$  is a  lucky permutation, and hence $i_0$ is a $\pi$-candidate, which implies $\pi(i_0)\in[k]$ and  consequently $\pi(i_0)\in \mathcal{X}_\pi$. Further, since $i_0$ is a $(\pi, l_0, w_0)$-element, we obtain $\rho_\pi(i_0)\leq l_0$, which means that the set $\pi([i_0])$ contains at most $l_0$ elements which are not greater than $\pi(i_0)$. Since $\pi(i_0)\in \mathcal{X}_\pi\subseteq [k]$ and $\mathcal{X}'_\pi\subseteq [k,n]$, all these elements must belong to the set $\mathcal{X}_\pi$. In particular, if we denote $\mathcal{X}_\pi:=\{y_1, \ldots, y_{j_0}\}$ for some $j_0\in [k]$, where $y_1<y_2<\ldots<y_{j_0}$, then we obtain: $\pi(i_0)=y_\iota$ for some  $1\leq \iota\leq \min\{j_0, l_0\}$.  Note that $\iota$ is the relative $\pi$-rank of the element $i_0$.

Let  $Y, Y'\subseteq [n]$ be arbitrary subsets which satisfy the following conditions
\begin{equation}\label{1}
Y\subseteq[k],\;\;\;Y'\subseteq [k,n],\;\;\;|Y|+|Y'|=i_0.
\end{equation}
Let us denote
\begin{eqnarray*}
S(Y, Y', l_0, i_0)&:=&\{\pi\in S(l_0, i_0)\colon \mathcal{X}_\pi=Y,\;\mathcal{X}'_\pi=Y'\},\\
\mu_{j, l_0}&:=&\min\{j, l_0\},\;\;\;j\in[k].
\end{eqnarray*}

\begin{proposition}\label{pomol1}
For every permutation $\pi\in S(Y, Y', l_0, i_0)$ the following three conditions hold:
\begin{itemize}
\item [(i)] $\pi([i_0])=Y\cup Y'$,
\item [(ii)] if $Y=\{y_1,y_2,\ldots,y_{j_0}\}$ for some $j_0\in[k]$ and $y_1<y_2<\ldots< y_{j_0}$, then there is $\iota\in[\mu_{j_0, l_0}]$ such that $\pi(i_0)=y_{\iota}$,
\item [(iii)] if $\pi([i_0-1])=\{y_1',y_2',\ldots,y'_{i_0-1}\}$ and $y_1'<y_2'<\ldots<y'_{i_0-1}$, then  $y'_j\in \pi([x_j])$ for every $j\in [l_0]$.
\end{itemize}
Conversely, if the sets $Y$, $Y'$ satisfy (\ref{1}) and a permutation $\pi\in Sym(n)$ satisfies (i)-(iii), then $\pi\in S(Y, Y', l_0, i_0)$.
\end{proposition}
\begin{proof}[of Proposition~\ref{pomol1}]
Let $\pi\in S(Y, Y', l_0, i_0)$ be arbitrary. By the above reasoning, the conditions (i)-(ii) directly follow from the equalities $Y=\mathcal{X}_\pi$ and $Y'=\mathcal{X}'_\pi$. To justify (iii) let us assume contrary that there is   $j_1\in[l_0]$ such that $i_1:=\pi^{-1}(y'_{j_1})\notin [x_{j_1}]$. Since $\pi(i_1)=y'_{j_1}\in \pi([i_0-1])$, we have $i_1<i_0$ and  $\pi([i_1])\subseteq \pi([i_0-1])=\{y_1',y_2',\ldots,y'_{i_0-1}\}$. Thus the set $\pi([i_1])$ contains at most $j_1$ elements which are not greater than $y'_{j_1}=\pi(i_1)$. Hence the relative $\pi$-rank of the element $i_1$ is not greater than $j_1$, i.e. $\rho_{\pi}(i_1)\leq j_1$. Since the sequence $(x_{j_1}, \ldots, x_{l_0+1})$ is non--decreasing, we have
$$
[x_{l_0+1}]=[x_{j_1}]\cup\bigcup_{j_1\leq j\leq l_0}[x_j, x_{j+1}].
$$
Since $i_1\in [i_0]\setminus [x_{j_1}]\subseteq [x_{l_0+1}]\setminus [x_{j_1}]$, there is $j_1\leq j_2\leq l_0$ such that $i_1\in [x_{j_2}, x_{j_2+1}]$. But  $\rho_{\pi}(i_1)\leq j_1\leq j_2$, and hence  $i_1$ is a $(\pi, j_2, w_0)$-element. Consequently $j_2$ is a $w_0$-threshold of $\pi$. Since $j_2\leq l_0$ and  $l_0$ is the smallest $w_0$-threshold of $\pi$, we obtain: $j_2=l_0$. Consequently, the element $i_1$ is a $(\pi, l_0, w_0)$-element. Since $i_1<i_0$, we obtain the contradiction with the assumption that  $i_0$ is the smallest $(\pi, l_0, w_0)$-element. This justifies the first part of Proposition~\ref{pomol1}.

Conversely, let $\pi\in Sym(n)$ be an arbitrary permutation which  satisfies (i)-(iii). We show that $\pi\in S(Y, Y', l_0, i_0)$.  By (ii), we have $\pi(i_0)\in Y\subseteq [k]$, and hence $i_0$ is a $\pi$-candidate. The equalities $Y=\mathcal{X}_\pi$ and $Y'=\mathcal{X}'_{\pi}$ directly follows from the definition of the sets $\mathcal{X}_\pi$, $\mathcal{X}'_\pi$ as well as from the conditions~(\ref{1}) and from (i). Next, we have $\pi(i_0)=y_\iota$ for some $\iota\in [\mu_{j_0, l_0}]$. Since $\pi([i_0])=Y\cup Y'$ and every element in $Y'$ is greater than every element in $Y$, we see by (ii) that $\{y_1, \ldots, y_\iota\}$ is the set of all elements from $\pi([i_0])$ which are  not greater than $y_\iota=\pi(i_0)$. Thus the relative $\pi$-rank of $i_0$ is equal to $\iota$. But $\iota\leq l_0$ and hence $\rho_{\pi}(i_0)\leq l_0$. Since $i_0\in [x_{l_0}, x_{l_0+1}]$, we see that  $i_0$ is a $(\pi, l_0, w_0)$-element. Thus $l_0$ is a $w_0$-threshold of $\pi$. To show that $l_0$ is the smallest $w_0$-threshold of $\pi$, suppose contrary that there is $l_1<l_0$ such that $l_1$ is a $w_0$-threshold of  $\pi$. Then there is $i_1\in [x_{l_1}, x_{l_1+1}]$ such that
\begin{equation}\label{mmmmm}
\rho_{\pi}(i_1)\leq l_1.
\end{equation}
By (iii) we have $y'_j\in \pi([x_j])$ for every $j\in[l_1]$. But for every $j\in[l_1]$ we have $[x_j]\subseteq [x_{l_1}]\subseteq [i_1]$. Thus we have
\begin{equation}\label{mmm}
\{y_1', \ldots, y_{l_1}', \pi(i_1)\}\subseteq \pi([i_1]).
\end{equation}
Since $i_0\in[x_{l_0}, x_{l_0+1}]$, $i_1\in [x_{l_1}, x_{l_1+1}]$ and $l_1<l_0$, we have $i_1<i_0$ and hence, by (iii), we have
$\pi([i_1])\subseteq \{y_1', \ldots, y_{i_0-1}'\}$. Thus there is $j_1\in[i_0-1]$ such that
\begin{equation}\label{mmmm}
\pi(i_1)=y'_{j_1}.
\end{equation}
Assuming $j_1\in[l_1]$,  we would have by (iii): $\pi(i_1)=y'_{j_1}\in \pi([x_{j_1}])$, and hence $i_1\in [x_{j_1}]\subseteq [x_{l_1}]$. But, this is impossible as $i_1\in [x_{l_1}, x_{l_1+1}]$. Thus it must be $j_1>l_1$, and consequently, we see  by (\ref{mmm})--(\ref{mmmm}) that the set $\pi([i_1])$ contains at least $l_1+1$ elements which are not grater than $\pi(i_1)$. Thus $\rho_{\pi}(i_1)\geq l_1+1$ and we have a contradiction with~ (\ref{mmmmm}). Hence $l_0$ is indeed the smallest $w_0$-threshold of $\pi$. To show that $\pi\in S(Y, Y', l_0, i_0)$, we now need to show that  $i_0$ is  the smallest $(\pi, l_0, w_0)$-element. Suppose contrary that there is $i_1<i_0$ such that $i_1$ is a ($\pi, l_0, w_0)$-element. We have $i_1\in [x_{l_0}, x_{l_0+1}]$ and $\rho_{\pi}(i_1)\leq l_0$. Similarly as above, we obtain by (iii) the inclusion $\{y'_1, \ldots, y'_{l_0}, \pi(i_1)\}\subseteq \pi([i_1])$. Since $i_1<i_0$, there is $j_1\in[i_0-1]$ such that $\pi(i_1)=y'_{j_1}$. Similarly as above, we  show that $j_1>l_0$. This implies $\rho_{\pi}(i_1)\geq l_0+1$ and we have a contradiction. Consequently $i_0$ is indeed the smallest $(\pi, l_0, w_0)$-element, and hence $\pi\in S(Y, Y', l_0, i_0)$. \qed
\end{proof}

\begin{proposition}\label{lolip}
The number of elements in the set $S(Y, Y', l_0, i_0)$ is equal to
$$
\mu_{j_0, l_0}\cdot (i_0-l_0-1)!\cdot (n-i_0)!\cdot \prod_{j=1}^{l_0} (x_j-j+1),
$$
where $j_0:=|Y|$.
\end{proposition}
\begin{proof}[of Proposition~\ref{lolip}]
We use the characterization of the set $S(Y, Y', l_0, i_0)$ from  Proposition~\ref{pomol1}. By the conditions (i)--(iii), we see that every permutation $\pi\in S(Y, Y', l_0, i_0)$ can be constructed as follows. At first, we choose arbitrarily an element $\iota\in[\mu_{j_0, l_0}]$ and  define: $\pi(i_0):=y_\iota$. We can do that in $\mu_{j_0, l_0}$ ways. Next,  for every $j\in [l_0]$ we choose  an element $i_j\in [x_j]$ and define $\pi(i_j):=y'_j$. We can do that in
$\prod_{j=1}^{l_0} (x_j-j+1)$ ways. Further, we define  the $\pi$-ranks   from the set
$$
\widetilde{Y}:=(Y\cup Y')\setminus\{y_\iota, y'_1, \ldots, y'_{l_0}\}=\{y'_{l_0+1}, \ldots, y'_{i_0-1}\}.
$$
Namely, for every $y\in \widetilde{Y}$ we choose an element $i_y\in [i_0-1]\setminus \{i_1, \ldots, i_{l_0}\}$ and define $\pi(i_y):=y$. We can do that in $(i_0-l_0-1)!$ ways. Finally, we define the $\pi$-ranks from the set $[n]\setminus (Y\cup Y')$, i.e. for every $i\in [n]\setminus [i_0]$ we choose an element $y_i\in [n]\setminus (Y\cup Y')$ and define $\pi(i):=y_i$. We can do that in $(n-i_0)!$ ways. Hence, the claim directly  follows from the above construction.\qed
\end{proof}

\begin{proposition}\label{lolip2}
The number of elements in the set $S(l_0, i_0)$ is equal to
$$
n!\cdot l_0\cdot r_{l_0}(i_0)\cdot \prod_{j=1}^{l_0} (x_j-j+1).
$$
\end{proposition}
\begin{proof}[of Proposition~\ref{lolip2}]
We have:
$$
|S(l_0, i_0)|=\sum_{j\in[k]}\sum_{(Y, Y')\in M_j}|S(Y, Y', l_0, i_0)|,
$$
where for every $j\in[k]$ we define
$$
M_j:=\{(Y, Y')\colon Y\subseteq[k],\;\;Y'\subseteq [n]\setminus[k],\;\;|Y|=j,\;\;|Y'|=i_0-j\}.
$$
Since $|M_j|={k\choose j}{n-k\choose i_0-j}$, we obtain from Proposition~\ref{lolip}:
\begin{equation}\label{eq3}
|S(l_0, i_0)|=(i_0-l_0-1)!\cdot (n-i_0)!\cdot \lambda(i_0, l_0)\cdot \prod_{j=1}^{l_0} (x_j-j+1),
\end{equation}
where
\begin{eqnarray*}
\lambda(i_0, l_0)&=&\sum_{j=1}^k{k\choose j}{n-k\choose i_0-j}\mu_{j, l_0}=\\
&=&\sum_{j=1}^{l_0}{k\choose j}{n-k\choose i_0-j} j+\sum_{j=l_0+1}^k{k\choose j}{n-k\choose i_0-j}l_0=\\
&=&\sum_{j=0}^{l_0}{k\choose j}{n-k\choose i_0-j} j+\sum_{j=0}^k{k\choose j}{n-k\choose i_0-j}l_0-\sum_{j=0}^{l_0}{k\choose j}{n-k\choose i_0-j}l_0=\\
&=&\sum_{j=0}^k{k\choose j}{n-k\choose i_0-j}l_0-\sum_{j=0}^{l_0}(l_0-j){k\choose j}{n-k\choose i_0-j}.
\end{eqnarray*}
From the Vandermonde's identity, we obtain
$$
\lambda(i_0, l_0)=l_0{n\choose i_0}-\sum_{j=0}^{l_0}(l_0-j){k\choose j}{n-k\choose i_0-j}=\frac{l_0\cdot n!}{(n-i_0)!\cdot (i_0-l_0-1)!}\cdot r_{l_0}(i_0).
$$
The claim  now follows from~(\ref{eq3}).\qed
\end{proof}

Obviously, the number of all lucky permutations from $\Pi_{w_0}$ which have  a $w_0$-threshold is equal to
$\sum_{l=1}^k\sum_{i=x_{l}+1}^{x_{l+1}}|S(l, i)|$. We see by Proposition~\ref{lolip2} that  this double sum can be written as follows
$$
n!\cdot\sum_{l=1}^k\left(\left(\sum\limits_{i=x_{l}+1}^{x_{l+1}}lr_l(i)\right)\cdot \prod_{j=1}^{l} (x_j-j+1)\right).
$$
The crucial point for the further study, which also finishes the proof of the first part of Theorem~\ref{tt1},  is the observation that the sum $\sum_{i=x_{l}+1}^{x_{l+1}}lr_l(i)$ in the above expression can be written in a closed form as follows:
$$
\sum_{i=x_{l}+1}^{x_{l+1}}lr_l(i)=r_{l-1}(x_l)-r_{l-1}(x_{l+1}).
$$
The last equality follows from the identity $lr_l(x)=r_{l-1}(x)-r_{l-1}(x-1)$ for all $l\in[\leq k]$ and $x\in [l,n]$, which we derive in
Section~\ref{sqa1} (see Lemma~\ref{propppp}~(iii) therein).

To show the second part of Theorem~\ref{tt1}, we provide the following characterization of all lucky permutations from $\Pi_{w_0}$ which have no $w_0$-thresholds.

\begin{proposition}\label{pomol3}
Let  $\pi\in \Pi_{w_0}$ be an arbitrary  lucky permutation without $w_0$-thresholds. Then the following two conditions hold:
\begin{itemize}
\item[(i)] there is $\iota\in[k]$ such that $\pi(n)=\iota$,
\item[(ii)] if $[k+1]\setminus\{\iota\}=\{y_1, \ldots, y_k\}$ and $y_1<y_2<\ldots<y_{k}$, then $y_j\in \pi([x_j])$ for every $j\in[k]$.
\end{itemize}
Conversely, if $\pi\in Sym(n)$ is an arbitrary permutation which satisfies (i)--(ii), then $\pi$  is a lucky permutation corresponding to $w_0$ and $\pi$ has no $w_0$-thresholds.
\end{proposition}
\begin{proof}[of Proposition~\ref{pomol3}]
The condition (i) directly follows from the definition of a lucky permutation. To show (ii), we proceed in the similar way as in the proof of Proposition~\ref{pomol1}. Namely, suppose contrary that there is $j_0\in[k]$ such that $i_0:=\pi^{-1}(y_{j_0})\notin [x_{j_0}]$. By~(i), we have $i_0\neq n$. Thus $i_0\in [n-1]$, and since
$$
[n-1]:=[x_{j_0}]\cup \bigcup_{j_0\leq j\leq k}[x_j, x_{j+1}],
$$
we obtain that there is $j_0\leq j_1\leq k$ such that $i_0\in [x_{j_1}, x_{j_1+1}]$. Since  $y_{j_0}\in [k+1]\setminus\{\iota\}$ and $\iota\notin \pi([i_0])$, every element in  $\pi([i_0])$ which is not grater than $y_{j_0}$ belongs to $[k+1]\setminus\{\iota\}$. Since $\pi(i_0)=y_{j_0}$ and
$[k+1]\setminus \{\iota\}=\{y_1, \ldots, y_k\}$, we see that the set of all  elements from $\pi([i_0])$ which are not grater than $\pi(i_0)$ is contained in the set $\{y_1, \ldots, y_{j_0}\}$. Thus $\rho_{\pi}(i_0)\leq j_0\leq j_1$. Consequently the element $i_0$ is a $(\pi, j_1, w_0)$-element. Thus the set $[x_{j_1}, x_{j_1+1}]$ contains a $(\pi, j_1, w_0)$-element, which means that $j_1$ is a $w_0$-threshold of $\pi$, contrary to our assumption.

Conversely, let $\pi\in Sym(n)$ be an arbitrary permutation which satisfies (i)--(ii). We show that $\pi$  is a lucky permutation corresponding to $w_0$ and $\pi$ has no $w_0$-thresholds. By (i), it is enough to show that $\pi$ has no $w_0$-thresholds. We proceed in the similar way as in the proof of Proposition~\ref{pomol1}. Namely, suppose contrary that $\pi$ has a $w_0$-threshold. Then there are $j_0\in [k]$ and $i_0\in [x_{j_0}, x_{j_0+1}]$ such that $\rho_{\pi}(i_0)\leq j_0$. By (ii), we have $y_j\in \pi([x_j])$ for every $j\in [j_0]$. But   $[x_j]\subseteq [x_{j_0}]\subseteq [i_0]$ for every $j\in [j_0]$. Hence
$\{y_1, \ldots, y_{j_0}, \pi(i_0)\}\subseteq \pi([i_0])$. If $\pi(i_0)\notin [k+1]$, then $\pi(i_0)>y_{j_0}$ and consequently $\rho_{\pi}(i_0)\geq j_0+1$. Thus, it must be $\pi(i_0)\in [k+1]\setminus\{\iota\}$ (note that $i_0\neq n$ and hence $\pi(i_0)\neq \iota$). By (ii), there is $j_1\in [k]$ such that $\pi(i_0)=y_{j_1}$. Assuming $j_1\in[j_0]$,  we would have $\pi(i_0)=y_{j_1}\in \pi([x_{j_1}])$. Hence $i_0\in [x_{j_1}]\subseteq [x_{j_0}]$, which is impossible as $i_0\in [x_{j_0}, x_{j_0+1}]$. Thus it must be $j_1>j_0$ and consequently  $\rho_{\pi}(i_0)\geq j_0+1$ contrary to our assumption. This finishes the proof of Proposition~\ref{pomol3}. \qed
\end{proof}

By using the conditions (i)-(ii) from Proposition~\ref{pomol3}, we see that every lucky permutation $\pi\in \Pi_{w_0}$ without $w_0$-thresholds can be constructed as follows. At first, we choose arbitrarily an element $\iota\in[k]$ and we define: $\pi(n):=\iota$. We can do that in $k$ ways. Next, we choose for every $j\in [k]$ an element $i_j\in [x_j]$ and define $\pi(i_j):= y_j$. This can be done in $\prod_{j=1}^{k} (x_j-j+1)$ ways. Finally, we define  the $\pi$-ranks from the set $[n]\setminus\{\iota, y_1, \ldots, y_{k}\}$, which can be done in $(n-k-1)!$ ways.  As a result of this construction, we obtain the required formula. This completes the proof of Theorem~\ref{tt1}.
 \qed
\end{proof}

\section{The  formula for an optimal sequence}\label{s4}

Let $T\colon X^{(k)}\to\mathbb{R}$ be the map defined for every $w=(x_1, \ldots, x_k)\in X^{(k)}$ as follows: $T(w):=P_1(w)+P_2(w)$, where
\begin{eqnarray*}
P_1(w)&:=&\frac{|\Pi_{w,1}|}{n!}=\sum\limits_{l=1}^k\left((r_{l-1}(x_l)-r_{l-1}(x_{l+1}))\cdot \prod\limits_{j=1}^l(x_j-j+1)\right),\\
P_2(w)&:=&\frac{|\Pi_{w,2}|}{n!}=\xi_{k,n}\cdot\prod\limits_{j=1}^k(x_j-j+1),
\end{eqnarray*}
and $\xi_{k,n}:=k(n-k-1)!/n!$.  We see by Theorem~\ref{tt1} that the probability of success for the policy described in  Proposition~\ref{p1} is equal to $T(w_0)$, where $w_0=(s_1, \ldots, s_k)$.

\begin{definition}\label{jjdef}
If  $f\colon Z\to\mathbb{R}$ is a map with $Z\subseteq X^{(l)}$ or $Z\subseteq X_l$ ($l\in[k]$), then we call an element $w_0\in Z$ such that $f(w_0)\geq f(w)$ for every $w\in Z$ an {\it optimal point} of this map.
\end{definition}

In particular, we see that if $w_0\in X^{(k)}$ is an optimal point of the map $T$, then $w_0$ is  an  optimal sequence (see Definition~\ref{ddddef}) if and only if it is a non--decreasing sequence. In this section, we show (see Proposition~\ref{trud}) that every optimal point of the map $T$ is indeed a non--decreasing sequence, which implies that every optimal point of $T$ is simultaneously an optimal sequence. We also derive the formula for an optimal point of $T$ (see Propositions~\ref{prop10},\ref{p8}). For this aim, we introduce  certain natural generalizations of this map. Namely, for every $l\in [k]_0$ we define the map  $T_l\colon X^{(k-l)}\to\mathbb{R}$  as follows:
\begin{equation}\label{uogt}
T_l(w):=\sum\limits_{j=1}^{k-l}R_{l,j}(w)\cdot\Gamma_{l,1,j}(w)+\xi_{k,n}\cdot\Gamma_{l,1,k-l}(w),
\end{equation}
where the maps $R_{l,j},\;\Gamma_{l,j,j'}\colon X^{(k-l)}\to\mathbb{R}$ ($l\in[k]_0$, $j,j'\in[n]_0$) are defined as follows (in the  formula for $R_{l,j}$ below we assume $x_{k-l+1}:=n-1$):
\begin{eqnarray*}
R_{l,j}((x_1, \ldots, x_{k-l}))&:=&\left\{
\begin{array}{ll}
r_{l+j-1}(x_j)-r_{l+j-1}(x_{j+1}),&{\rm if}\;\;1\leq j\leq k-l,\\
0,&{\rm otherwise},
\end{array}
\right.\\
\Gamma_{l,j,j'}((x_1, \ldots, x_{k-l}))&:=&\left\{
\begin{array}{ll}
\prod\limits_{t=j}^{j'}(x_t-l-t+1),&{\rm if}\;\;1\leq j\leq j'\leq k-l,\\
1,&{\rm otherwise}.
\end{array}
\right.
\end{eqnarray*}
In particular $T_0=T$ and $T_k=\xi_{k,n}$.

Let us consider the maps  $D_l$ ($l\in[k]_0$) defined by (\ref{dddd}). In Section~\ref{sqa1}, we derive some properties of these maps  (see Lemma~\ref{prop8} therein), which allows to describe them  in terms of the maps $d_l$ defined by (\ref{aj}) in the following way.

\begin{proposition}\label{proo8}
For every $w=(x_1, \ldots x_k)\in X^{(k)}$ and every $l\in[k]_0$ we have
$$
D_l(\sigma^l(w))=\sum_{i=l}^{k-1}\left(\prod_{j=l+1}^i \widetilde{x}_j\right)d_i(\widetilde{x}_{i+1})+\left(\prod_{j=l+1}^k \widetilde{x}_j\right)\delta_{k,n},
$$
where  $\sigma\colon X^{(k-l)}\to X^{(k-l-1)}$ is a left-shift operator removing the first letter from a non-empty word and $\widetilde{x}_i:=x_i-i+1$ for every $i\in [k]$.
\end{proposition}
\begin{proof}
The case $l=k$ directly follows from the equality $D_k(\epsilon)=\delta_{k,n}$ (see Lemma~\ref{prop8}~(i)). In the case $l\in[k-1]_0$, we have by  Lemma~\ref{prop8}~(ii)
$$
D_l(\sigma^l(w))=c_l(x_{l+1})+(x_{l+1}-l)D_{l+1}(\sigma^{l+1}(w)),
$$
where the map $c_l$ is defined by (\ref{cccd}). Hence, since $c_l(x_{l+1})=d_l(\widetilde{x}_{l+1})$, we can write
$$
D_l(\sigma^l(w))=d_l(\widetilde{x}_{l+1})+\widetilde{x}_{l+1}D_{l+1}(\sigma^{l+1}(w)).
$$
By easy induction on $m$, we can extend the last formula as follows:
\begin{equation}\label{erex}
D_l(\sigma^l(w))=\sum_{i=l}^m\left(\prod_{j=l+1}^i\widetilde{x}_j\right)d_i(\widetilde{x}_{i+1})+\left(\prod_{j=l+1}^{m+1}\widetilde{x}_j\right)D_{m+1}(\sigma^{m+1}(w))
\end{equation}
for every $l\in[k-1]_0$ and  $m\in[l-1, k-1]$. The claim now follows by taking $m:=k-1$ in~(\ref{erex}).\qed
\end{proof}

Let  $F_{l, w}$ ($l\in[k]$, $w\in X^{(k-l)}$) be the maps defined by (\ref{fffd}). In the proof of the next proposition, we use some properties of the maps $c_l$, which we  derive in Section~\ref{sqa1}.

\begin{proposition}\label{prop10}
Let $l\in[k]$ and $w\in X^{(k-l)}$. If $s$ is the smallest number $x\in[l-1, n-1]$ such that $c_l(x+1)\leq D_l(w)/l$, then $s$ is an optimal point of the map $F_{l,w}$. Consequently, if $t$ is the smallest number $x\in[n-l]$ such that $d_l(x)\leq D_l(w)/l$, then $t+l-1$ is an optimal point of  $F_{l,w}$.
\end{proposition}
\begin{proof}
For every $x\in [l-1,n-2]$ the following equality holds:
\begin{equation}\label{uss}
F_{l,w}(x+1)-F_{l,w}(x)=l\cdot c_{l}(x+1)-D_l(w).
\end{equation}
Indeed, directly by the definition of the map $F_{l,w}$, the left side of (\ref{uss}) is equal to
$c_{l-1}(x)-c_{l-1}(x+1)-D_l(w)$, which, by the identity $lc_l(x+1)=c_{l-1}(x)-c_{l-1}(x+1)$ for all $x\in [l-1,n-1]$ (see Lemma~\ref{prop7}~(iii) in Section~\ref{sqa1}), is equal to the right side of (\ref{uss}).  Hence, the first part follows from~(\ref{uss}) and from  the fact that the map $c_l$ is non--increasing (see Lemma~\ref{prop7}~(iv)). The  second part follows now from the equalities $d_l(x)=c_l(x+l)$ for all $x\in [n-l]$.
\qed
\end{proof}

The below proposition is based on the observation that for every word $w=(x_1, \ldots, x_k)\in X^{(k)}$ and every $l\in[k]$ there are $A\in \mathbb{R}_+$, $B\in \mathbb{R}$ which do not depend on the letter $x_l$ and such that $T(w)=A\cdot F_{l, \sigma^l(w)}(x_l)+B$ (for the proof see Lemma~\ref{prop9} in Section~\ref{sqa1}).

\begin{proposition}\label{p8}
A sequence $w=(x_1, \ldots, x_k)\in X^{(k)}$ is an optimal point of the map $T$ if and only if for every $l\in[k]$ the letter $x_l$ is an optimal point of the map $F_{l, \sigma^l(w)}$.
\end{proposition}
\begin{proof}
Suppose, contrary, that the sequence $w=(x_1, \ldots, x_k)$ is an optimal point of $T$ and there is $l\in[k]$ such that the letter $x_l$ is not an optimal point of $F_{l, \sigma^l(w)}$. Let $x'_l\in X_l$ be an optimal point of $F_{l, \sigma^l(w)}$ and let $w'\in X^{(k)}$ be the word arising from $w$ by replacing the $l$-th coordinate with $x'_l$. Since $\sigma^l(w')=\sigma^l(w)$, we have:
$$
F_{l, \sigma^l(w')}(x'_l)=F_{l, \sigma^l(w)}(x'_l)>F_{l, \sigma^l(w)}(x_l).
$$
By Lemma~\ref{prop9}, there are $A\in \mathbb{R}_+$, $B\in\mathbb{R}$ such that
$$
T(w')=A\cdot F_{l, \sigma^l(w')}(x'_l)+B,\;\;\;T(w)=A\cdot F_{l, \sigma^l(w)}(x_l)+B.
$$
Consequently $T(w')>T(w)$, which contradicts with the assumption that $w$ is an optimal point of $T$.

Conversely, let $w=(x_1,\ldots, x_k)\in X^{(k)}$ be such that  for every $l\in[k]$ the letter $x_l$ is an optimal point of the map $F_{l, \sigma^l(w)}$. We show that $w$ is an optimal point of $T$. Let $v=(y_1, \ldots, y_k)\in X^{(k)}$ be an arbitrary optimal point of the map $T$. Let us define the words  $w_l\in X^{(k)}$ ($0\leq l\leq k$)  as follows: $w_0:=w$, $w_k:=v$ and $w_l:=(y_1, \ldots, y_l, x_{l+1},\ldots, x_k)$ for every $l\in [k-1]$. In particular, for each $l\in[k]$ the two words $w_{l-1}$ and $w_l$ differ only in the $l$-th position, which is equal to $x_l$ in $w_{l-1}$ and to $y_l$ in $w_l$. Hence, by Lemma~\ref{prop9}, there are $A\in\mathbb{R}_+$ and $B\in\mathbb{R}$ such that
$$
T(w_{l-1})=A\cdot F_{l, \sigma^l(w_{l-1})}(x_l)+B,\;\;\;T(w_l)=A\cdot F_{l, \sigma^l(w_l)}(y_l)+B.
$$
Since the letter $x_l$ is an optimal point of $F_{l, \sigma^l(w)}$ and   $\sigma^l(w_{l-1})=\sigma^l(w_l)=\sigma^l(w)$, we have $T(w_{l-1})\geq T(w_l)$. Consequently, we obtain the inequalities:
$$
T(w)=T(w_0)\geq T(w_1)\geq \ldots\geq T(w_k)=T(v).
$$
Since  $v$ is an optimal point of $T$, we have $T(w)=T(v)$. Thus $w$ is an optimal point of $T$.
\qed
\end{proof}

The proof of the next proposition is based on various properties of the maps $D_l$, $c_l$ ($l\in[k]_0$), which we derive in Section~\ref{sqa1} (see
 Lemmas~\ref{prop7},~\ref{prop8}).

\begin{proposition}\label{trud}
If $w=(x_1, \ldots, x_k)\in X^{(k)}$ is an optimal point of the map $T$, then $w$ is a non--decreasing sequence.
\end{proposition}
\begin{proof}
Let $w=(x_1, \ldots, x_k)$ be  an optimal point of the map $T$. By Proposition~\ref{p8}, we see that for every $l\in[k]$ the letter $x_l$ is an optimal point of the map $F_{l, \sigma^l(w)}$. Let us fix $l\in[k]\setminus\{1\}$ and let us denote $x:=x_{l-1}$, $y:=x_l$, $v:=\sigma^l(w)$. We have to show that $x\leq y$. We can assume that $x\neq l-1$ and $y\neq n-1$. Since $x$ and $y$ are optimal points of $F_{l-1 yv}$ and $F_{l, v}$, respectively, we obtain:
$$
F_{l,v}(y+1)-F_{l,v}(y)\leq 0, \;\;\;F_{l-1, yv}(x)-F_{l-1,yv}(x-1)\geq 0.
$$
By~(\ref{uss}), we have
\begin{equation}\label{w1}
l\cdot c_l(y+1)\leq D_l(v),\;\;\;(l-1)\cdot c_{l-1}(x)\geq  D_{l-1}(yv).
\end{equation}
Since $y-l+1>0$, we obtain from the first of the  inequalities in~(\ref{w1}):
\begin{equation}\label{wwww1}
(y-l+1)\cdot l\cdot c_l(y+1)\leq (y-l+1)\cdot D_l(v).
\end{equation}
But from Lemma~\ref{prop8}~(ii), we have
\begin{equation}\label{w2}
(y-l+1)\cdot D_l(v)=D_{l-1}(yv)-c_{l-1}(y).
\end{equation}
The second inequality in~(\ref{w1}) gives:
\begin{equation}\label{www1}
D_{l-1}(yv)-c_{l-1}(y)\leq (l-1)\cdot c_{l-1}(x)-c_{l-1}(y).
\end{equation}
From~(\ref{wwww1})--(\ref{www1}), we obtain:
\begin{equation}\label{ww1}
(y-l+1)\cdot l\cdot c_l(y+1)\leq (l-1)\cdot c_{l-1}(x)-c_{l-1}(y).
\end{equation}
Since  $l\cdot c_{l}(y+1)=c_{l-1}(y)-c_{l-1}(y+1)$ (see Lemma~\ref{prop7}~(iii)), we obtain
$$
(y-l+1)\cdot (c_{l-1}(y)-c_{l-1}(y+1))\leq (l-1)\cdot c_{l-1}(x)-c_{l-1}(y),
$$
or equivalently:
\begin{equation}\label{lin}
(l-1)\cdot (c_{l-1}(x)-c_{l-1}(y+1))\geq (y-l+2)\cdot c_{l-1}(y)-y\cdot c_{l-1}(y+1).
\end{equation}
But the right side of~(\ref{lin}) is equal to $\frac{{n-y-1\choose k-l+1}}{(l-1)!\cdot {n\choose k}}$ (see Lemma~\ref{prop7}~(v)), which in the case $y\leq n-k+l-2$ is a positive number.
Consequently, we have $c_{l-1}(x)-c_{l-1}(y+1)>0$ in this case. Since the map $c_{l-1}$ is non--increasing (see Lemma~\ref{prop7}~(iv)), we have $x\leq y$. So,  we can  assume  $y\geq n-k+l-1$. In the case $l\geq 3$ we have $D_{l-1}(yv)>0$ by Lemma~\ref{prop8}~(iii), and hence, by~(\ref{w1}), we obtain in this case: $c_{l-1}(x)>0$. But then, directly from the definition of the map $c_{l-1}$, we have
$$
c_{l-1}(x)=\frac{k\cdot (x-l+1)!\cdot (n-x)!}{(l-1)\cdot (l-2)\cdot n!}\sum\limits_{j=0}^{l-3}{k-1\choose j}{n-k\choose x-j-1}>0.
$$
Thus there must be $j\in[l-3]_0$ such that $n-k\geq x-1-j$. Consequently $x\leq n-k+1+j\leq n-k+l-2<y$. Hence, we can assume $l=2$. Then by~(\ref{w2}), we have $D_1(yv)-c_1(y)=(y-1)\cdot D_2(v)$ and by  the second inequality in~(\ref{w1}), we have $c_1(x)\geq D_1(yv)$. Hence
$c_1(x)-c_1(y)\geq (y-1)\cdot D_2(v)>0$, where the last inequality follows from Lemma~\ref{prop8}~(iii). Since the map $c_1$ is non--increasing, we obtain $x<y$, which finishes the proof.\qed
\end{proof}

\section{The proofs of Theorems~\ref{mt}-\ref{mtt}}\label{secpr}

The main results are a straightforward consequence of Theorem~\ref{tt1} and Propositions~\ref{proo8}-\ref{trud}.

\begin{proof}[of Theorem~\ref{mt}]
Let $w_0:=(s_1, \ldots, s_k)$ be a sequence constructed as in Theorem~\ref{mt}. By Proposition~\ref{proo8}, we see that for every $l\in[k]$ the right side of~(\ref{ineqqq}) is equal to $D_l(\sigma^l(w_0))/l$. Thus  for every $l\in[k]$ the number $t_l=s_l-l+1$ is the smallest number $x\in[n-l]$ which satisfies $d_l(x)\leq D_l(\sigma^l(w_0))/l$, and hence, by Proposition~\ref{prop10},  the number $s_l$  is an optimal point of the map $F_{l, \sigma^l(w_0)}$. By Proposition~\ref{p8}, we obtain that $w_0$ is an optimal point of the map $T$. Moreover, the sequence $w_0$ is non-decreasing by Proposition~\ref{trud}. Hence, we see by the definition of the map $T$ and by Theorem~\ref{tt1} that $w_0$ is an optimal sequence. \qed
\end{proof}

\begin{proof}[of Theorem~\ref{mtt}]
Let us denote $w_0:=(s_1, \ldots, s_k)$. By Proposition~\ref{proo8}, the expression~(\ref{expr}) is equal to $-D_0(w_0)$, which, by the definition of the map $D_0$, is equal to $T_0(w_0)=T(w_0)$. Hence, by the definition of the map $T$ and by Theorem~\ref{tt1}, this expression is equal to the probability of success for the policy described in Proposition~\ref{p1}.\qed
\end{proof}

\section{The auxiliary properties of the maps $r_l$, $c_l$, $D_l$ and $F_{l, w}$}\label{sqa1}

Let us  define for every $l\in[k]$ the maps $a_l,b_l\colon [n]\to\mathbb{R}$ as follows:
$$
a_l(x):=\frac{1}{{n \choose x}}\sum\limits_{j=0}^l{k\choose j}{n-k\choose x-j},\;\;\;b_l(x):=\frac{1}{l{n \choose x}}\sum\limits_{j=0}^lj{k\choose j}{n-k\choose x-j}.
$$

\begin{lemma}\label{prop1}
For all $x\in[n]$, $l\in[k-1]$ we have:
\begin{eqnarray}
a_{l+1}(x)-a_l(x)&=&\gamma(x,l+1),\label{ajj1}\\
(l+1)\cdot b_{l+1}(x)-l\cdot b_l(x)&=&(l+1)\cdot\gamma(x,l+1)\label{ajj2},
\end{eqnarray}
and for all $x\in[n-1]$ and $l\in[k]$ we have:
\begin{eqnarray}
a_l(x)-a_{l}(x+1)&=&\frac{k-l}{n-x}\cdot\gamma(x,l),\label{ajj3}\\
(x+1)\cdot b_l(x)-x\cdot b_l(x+1)&=&\frac{(x+1)\cdot(k-l)}{n-x}\cdot\gamma(x,l)\label{bi},
\end{eqnarray}
where $\gamma(x,l):={k\choose l}{n-k\choose x-l}/{n\choose x}$. In particular, for all $x\in[1,n]$ and $l\in[1,k]$ we have
\begin{eqnarray}
x\cdot a_{l-1}(x-1)-(x-l)\cdot a_{l-1}(x)&=&l\cdot a_{l}(x),\label{acor}\\
x\cdot b_{l-1}(x-1)-(x-l)\cdot b_{l-1}(x)&=&l\cdot b_{l}(x)\label{bcor}.
\end{eqnarray}
\end{lemma}
\begin{proof}
The  identities~(\ref{ajj1})-(\ref{ajj2})  directly follow from the definitions of the maps $a_l$, $b_l$; the  identities~(\ref{ajj3})-(\ref{bi}) can be easily  proved by induction on $l$. These four identities together with the definitions of the maps $a_l$, $b_l$ imply the identities (\ref{acor})--(\ref{bcor}).\qed
\end{proof}

\begin{lemma}\label{propppp}
The maps $r_l$ have the following properties:
\begin{itemize}
\item[(i)] $r_l(x)=\frac{(x-l-1)!}{x!}\cdot\left(b_l(x)-a_l(x)+1\right)$ for  $l\in[k]$ and $x\in [l,n]$,
\item[(ii)] $r_0(x)=-\sum_{j=2}^x r_1(j)$ for $x\in[n]$,
\item[(iii)] $l\cdot r_l(x)=r_{l-1}(x-1)-r_{l-1}(x)$ for $l\in[\leq k]$ and $x\in [l,n]$.
\end{itemize}
\end{lemma}
\begin{proof}
The item (i) follows directly from the definition of the map $r_l$. To show (ii), we can write by the definition of  $r_1$:
$$
r_1(j)=\left(\frac{1}{j-1}-\frac{1}{j}\right)-\frac{(n-j)!\cdot (j-2)!}{n!}\cdot{n-k\choose j},\;\;j\in[1,n].
$$
Hence, by the following easily verifiable identity
$$
\frac{(n-j)!\cdot (j-2)!}{n!}\cdot{n-k\choose j}=\frac{1}{{n\choose k}}
\left(
\frac{{n-j\choose k}}{j-1}-\frac{{n-(j+1)\choose k}}{j}\right)-\frac{{n-j-1\choose k-1}}{j{n\choose k}},
$$
we obtain:
$$
r_1(j)=\left(\frac{1}{j-1}-\frac{1}{j}\right)+\frac{{n-j-1\choose k-1}}{j{n\choose k}}-
\frac{1}{{n\choose k}}
\left(
\frac{{n-j\choose k}}{j-1}-\frac{{n-(j+1)\choose k}}{j}\right),\;\;j\in[1,n].
$$
Hence, the item (ii) in the case $x\in[1,n]$ simply follows from the above equalities. The case  $x=1$ can be directly verified. The item~(iii) in the case $l\leq 0$  simply follows from the definition of the map $r_l$. In the case $l=1$ it follows from the item~(ii).
If $l>1$, then for every $x\in[l,n]$ we can write by the item (i):
\begin{eqnarray*}
r_{l-1}(x-1)&=&\frac{(x-l-1)!}{(x-1)!}(b_{l-1}(x-1)-a_{l-1}(x-1)+1),\\
r_{l-1}(x)&=&\frac{(x-l)!}{x!}\left(b_{l-1}(x)-a_{l-1}(x)+1\right).
\end{eqnarray*}
Now, we can use the equalities~(\ref{acor})--(\ref{bcor}) from Lemma~\ref{prop1} and obtain that the difference $r_{l-1}(x-1)-r_{l-1}(x)$ is equal to
$$
\frac{(x-l-1)!}{x!}\left(l\cdot b_{l}(x)-l\cdot a_{l}(x)+l\right)=l\cdot r_{l}(x),
$$
which finishes the proof of Lemma~\ref{propppp}.
\qed
\end{proof}

\begin{lemma}\label{prop7}
The maps $c_l$ ($l\in[k]_0$) have the following properties:
\begin{itemize}
\item[(i)] $c_{l}(x)=\frac{b_{l-1}(x)}{l\cdot l! \cdot{x\choose l}}$ for $l\in[1,k]$, $x\in[l,n]$,
\item[(ii)] $c_l(x)=r_{l-1}(x)-(x-l)\cdot r_l(x)$ for  $l\in[k]_0$, $x\in[l,n]$,
\item[(iii)] $l\cdot c_{l}(x)=c_{l-1}(x-1)-c_{l-1}(x)$ for  $l\in[k]$, $x\in [l,n]$,
 \item[(iv)] for every $l\in[k]$ the map $c_l$ is non--increasing,
\item[(v)] $(x+1-l)\cdot c_l(x)-x\cdot c_l(x+1)=\frac{{n-x-1\choose k-l}}{l!\cdot {n\choose k}}$ for $l\in[k]$ and $x\in [l,n-1]$.
\end{itemize}
\end{lemma}
\begin{proof}
By the definition of the maps $c_l$, we can write
$$
c_l(x)=\left\{
\begin{array}{ll}
-1+\frac{k\cdot x}{n}+\frac{{n-x-1\choose k}}{{n\choose k}}+\frac{x}{{n\choose k}}\sum\limits_{j=1}^{x}\frac{{n-j-1\choose k-1}}{j},&{\rm  if}\;\;l=0,\\
\\
-\frac{k}{n}-\frac{1}{{n\choose k}}\sum\limits_{j=1}^{x-1}\frac{{n-j-1\choose k-1}}{j},&{\rm  if}\;\;l=1,\\
\\
\frac{k\cdot (x-l)!\cdot (n-x)!}{l\cdot (l-1)\cdot n!}\sum\limits_{j=0}^{l-2}{k-1\choose j}{n-k\choose x-j-1},&{\rm  if}\;\;1<l\leq k.
\end{array}
\right.
$$
Hence, the item (i) easily follows from the definition of the map $b_{l-1}$ and from the identity $j{k\choose j}=k{k-1\choose j-1}$ for $j\in\mathbb{Z}$. The item~(ii) in the case $l\in\{0,1\}$ directly follows  from the definitions of the maps $c_l$, $r_{l-1}$, $r_l$. In the case $l\in[1,k]$, we can use the item (i) for the left side and Lemma~\ref{propppp}~(i) for the right side, and then the claim easily follows from the identities~(\ref{ajj1})-(\ref{ajj2}) from Lemma~\ref{prop1}. To show the item~(iii), we see by (ii) that the difference $c_{l-1}(x-1)-c_{l-1}(x)$ is equal to:
$$
(r_{l-2}(x-1)-r_{l-2}(x))-(x-l)\cdot r_{l-1}(x-1)+(x-l+1)\cdot r_{l-1}(x).
$$
By  Lemma~\ref{propppp}~(iii), we have $r_{l-2}(x-1)-r_{l-2}(x)=(l-1)\cdot r_{l-1}(x)$. Hence
$$
c_{l-1}(x-1)-c_{l-1}(x)=x\cdot r_{l-1}(x)-(x-l)\cdot r_{l-1}(x-1).
$$
Again, by Lemma~\ref{propppp}~(iii), we have $r_{l-1}(x-1)=r_{l-1}(x)+l\cdot r_l(x)$, and hence
$$
c_{l-1}(x-1)-c_{l-1}(x)=l(r_{l-1}(x)-(x-l)\cdot r_l(x)).
$$
The claim now follows from the item (ii). To show~(iv), we obtain by the item~(iii)  that for every  $l\in[k-1]$ and $x\in [l+1,n]$ the difference $c_l(x-1)-c_l(x)$ is equal to $(l+1)c_{l+1}(x)$, which is a nonnegative number  by the definition of the map $c_{l+1}$. Hence the map $c_l$ is non--increasing for every $l\in[k-1]$. If $l=k=1$,  then the item~(iv) follows directly from the definition of the map $c_l$. If $l=k>1$, then by the definition of the map $c_l$, we obtain
$$
c_k(x)=\frac{(x-k)!}{n(k-1)(x-1)!}-\frac{(n-k)!}{(k-1)n!}.
$$
Hence, we see  that also in this case the map $c_l$ is  non-increasing. As for the  item~(v), the case $l=1$ follows directly from the item~(iii) and from the definition of the map $c_2$. In the case $l\in[1,k]$, by  the item~(i), we obtain  that the difference $(x+1-l)c_l(x)-x c_l(x+1)$ is equal to
$$
\frac{(x-l+1)!}{l (x+1)!}\left((x+1)b_{l-1}(x)-x b_{l-1}(x+1)\right),
$$
which is equal to $\frac{{n-x-1\choose k-l}}{l!\cdot {n\choose k}}$ by the equality~(\ref{bi}) from Lemma~\ref{prop1}.
\qed
\end{proof}

\begin{lemma}\label{prop8}
The maps $D_l$ ($l\in[k]_0$) have the following properties:
\begin{itemize}
\item[(i)]  $D_k(\epsilon)=\delta_{k,n}$,
\item[(ii)] $D_l(w)=(x_1-l)\cdot D_{l+1}(\sigma(w))+c_{l}(x_1)$ for each  $l\in[k-1]_0$ and $w\in X^{(k-l)}$, where $x_1$ denotes the first letter of $w$,
\item[(iii)] $D_l(w)>0$ for each $l\in[1,k]$ and  $w\in X^{(k-l)}$.
\end{itemize}
\end{lemma}
\begin{proof}
The item~(i)  directly follows from the definition of the map $D_k$ and from the formulae~(\ref{ajj}) and~(\ref{rere}) defining, respectively, the number $\delta_{k,n}$ and the map $r_{k-1}$. As for the item~(ii), for every $l\in[k-1]_0$ we obtain by Lemma~\ref{prop7}~(ii) and by the definitions of the maps $R_{l,1}$, $\Gamma_{l,1,1}$:
\begin{eqnarray*}
D_l(w)=r_{l-1}(x_1)-T_l(w)=\\
=r_{l-1}(x_1)-R_{l,1}(w)\cdot\Gamma_{l,1,1}(w)-\left(\sum\limits_{j=2}^{k-l}R_{l,j}(w)\cdot\Gamma_{l,1,j}(w)+\xi_{k,n}\cdot\Gamma_{l,1,k-l}(w)\right)=\\
=c_l(x_1)+r_l(x_2)\cdot (x_1-l)-\left(\sum\limits_{j=2}^{k-l}R_{l,j}(w)\cdot\Gamma_{l,1,j}(w)+\xi_{k,n}\cdot\Gamma_{l,1,k-l}(w)\right),
\end{eqnarray*}
where $x_2$ denotes the second letter of $w$ in the case $l<k-1$ and $x_2:=n-1$ in the case $l=k-1$. In particular, if $l=k-1$, then  $\Gamma_{l,1,k-l}(w)=\Gamma_{k-1,1,1}(w)=x_1-l$, and hence
\begin{eqnarray*}
D_l(w)&=&c_l(x_1)+(x_1-l)\cdot (r_{k-1}(n-1)-\xi_{k,n})=\\
&=&c_l(x_1)+(x_1-l)\cdot D_k(\epsilon)=c_l(x_1)+(x_1-l)\cdot D_{l+1}(\sigma(w)).
\end{eqnarray*}
If $l<k-1$, then $\Gamma_{l,1,j}(w)=(x_1-l)\cdot \Gamma_{l,2,j}(w)$ for every $j\in[k-l]$, and hence
$$
D_l(w)=c_l(x_1)+r_l(x_2)\cdot (x_1-l)-(x_1-l)\cdot \Lambda,
$$
where
$$
\Lambda:=\sum\limits_{j=2}^{k-l}R_{l,j}(w)\cdot\Gamma_{l,2,j}(w)+\xi_{k,n}\cdot\Gamma_{l,2,k-l}(w).
$$
Since $R_{l,j}(w)=R_{l+1,j-1}(\sigma(w))$ and $\Gamma_{l,2,j}(w)=\Gamma_{l+1,1,j-1}(\sigma(w))$ for every $j\in[1,k-l]$, we obtain
\begin{eqnarray*}
\Lambda=\sum\limits_{j=2}^{k-l}R_{l+1,j-1}(\sigma(w))\cdot\Gamma_{l+1,1,j-1}(\sigma(w))+\xi_{k,n}\cdot\Gamma_{l+1,1,k-l-1}(\sigma(w))=\\
=\sum\limits_{j=1}^{k-l-1}R_{l+1,j}(\sigma(w))\cdot\Gamma_{l+1,1,j}(\sigma(w))+\xi_{k,n}\cdot\Gamma_{l+1,1,k-l-1}(\sigma(w))=T_{l+1}(\sigma(w)).
\end{eqnarray*}
Consequently, we have:
\begin{eqnarray*}
D_l(w)=c_l(x_1)+(x_1-l)\cdot r_l(x_2)-(x_1-l)\cdot T_{l+1}(\sigma(w))=\\
=c_l(x_1)+(x_1-l)(r_l(x_2)-T_{l+1}(\sigma(w)))=c_l(x_1)+(x_1-l)\cdot D_{l+1}(\sigma(w)),
\end{eqnarray*}
which finishes the proof of the item~(ii).  The item~(iii) directly follows from the items~(i)--(ii) and from the inequalities $c_l(x)\geq 0$ for all $l\in[1,k]$, $x\in[l,n]$.
\qed
\end{proof}

\begin{lemma}\label{prop9}
Let $w=(x_1, \ldots, x_k)\in X^{(k)}$ be arbitrary. Then for every $l\in[k]$ there exist $A\in\mathbb{R}_+$, $B\in\mathbb{R}$ which do not depend on the letter $x_l$ and such that $T(w)=A\cdot F_{l, \sigma^l(w)}(x_l)+B$.
\end{lemma}
\begin{proof}
Let us fix $l\in[k]$ and let us  define $A:=\Gamma_{0,1,l-1}(w)$. Then we see by the definition of the map $\Gamma_{0,1,l-1}$ that $A$ does not depend on $x_l$ and $A>0$. By the definition of the map $T=T_0$, we have $T(w)=B_1+B_2+B_3$, where
\begin{eqnarray*}
B_1&:=&\sum_{j=1}^{l-2}R_{0,j}(w)\cdot\Gamma_{0,1,j}(w),\\
B_2&:=&\sum_{j=l-1}^lR_{0,j}(w)\Gamma_{0,1,j}(w)=AR_{0,l-1}(w)+R_{0,l}(w)\Gamma_{0,1,l}(w),\\
B_3&:=&\sum_{j=l+1}^{k}R_{0,j}(w)\Gamma_{0,1,j}(w)+\xi_{k,n}\Gamma_{0,1,k}(w).
\end{eqnarray*}
By the definitions of the maps $R_{0,j}$, $\Gamma_{0,1,j}$ ($j\in[l-2]$), we see that $B_1$ does not depend on $x_l$. Since $\Gamma_{0,1,j}(w)=A(x_l-l+1)\Gamma_{0,l+1,j}(w)$ for every $j\in[l-1,k]$, we obtain $T(w)=B_1+AB_2'+AB_3'$, where
\begin{eqnarray*}
B_2'&:=&R_{0, l-1}(w)+(x_l-l+1)\cdot R_{0,l}(w),\\
B_3'&:=&(x_l-l+1)\cdot B_4,\\
B_4&:=&\sum\limits_{j=l+1}^k R_{0,j}(w)\cdot\Gamma_{0,l+1,j}(w)+\xi_{k,n}\cdot \Gamma_{0,l+1,k}(w).
\end{eqnarray*}
Next, by the definitions of the maps $R_{0,l-1}$, $R_{0,l}$ and by  Lemma~\ref{prop7}~(ii), we obtain
$$
B_2'=B_2''-c_{l-1}(x_l)-x_l\cdot r_{l-1}(x_{l+1}),
$$
where $B_2'':=r_{l-2}(x_{l-1})+(l-1)\cdot r_{l-1}(x_{l+1})$ does not depend on $x_l$. Further, since $R_{0,j}(w)=R_{l, j-l}(\sigma^l(w))$ and $\Gamma_{0,l+1,j}(w)=\Gamma_{l,1, j-l}(\sigma^l(w))$ for every $j\in[l-1,k]$, we obtain
\begin{eqnarray*}
B_4&=&\sum\limits_{j=l+1}^k R_{l,j-l}(\sigma^l(w))\cdot\Gamma_{l,1,j-l}(\sigma^l(w))+\xi_{k,n}\cdot \Gamma_{l,1,k-l}(\sigma^l(w))=\\
&=&\sum\limits_{j=1}^{k-l} R_{l,j}(\sigma^l(w))\cdot\Gamma_{l,1,j}(\sigma^l(w))+\xi_{k,n}\cdot \Gamma_{l,1,k-l}(\sigma^l(w))=T_l(\sigma^l(w)).
\end{eqnarray*}
Hence
$$
B_3'=(x_l-l+1)\cdot B_4=x_l\cdot T_l(\sigma^l(w))-B_3'',
$$
where $B_3'':=(l-1)\cdot T_l(\sigma^l(w))$ does not depend on $x_l$. We can now write
\begin{eqnarray*}
T(w)&=&B_1+A\cdot (B_2'+B_3')=\\
&=&B_1+A\cdot (B_2''-c_{l-1}(x_l)-x_l\cdot r_{l-1}(x_{l+1})+x_l\cdot T_l(\sigma^l(w))-B_3'')=\\
&=&B_1+A\cdot (B_2''-c_{l-1}(x_l)-x_l\cdot D_l(\sigma^l(w))-B_3'')=\\
&=&B_1+A\cdot(B_2''-B_3'')+A\cdot F_{l,\sigma^l(w)}(x_l)=B+A\cdot F_{l,\sigma^l(w)}(x_l),
\end{eqnarray*}
where $B:=B_1+A\cdot(B_2''-B_3'')$ does not depend on $x_l$.\qed
\end{proof}

\section{Conclusion}

In the present paper, we obtained the analytic formulae for an optimal sequence $(s_1, \ldots, s_k)$ in the Gusein-Zade version (\cite{1}) of a generalized secretary problem. In this  problem, the interviewer would like to choose one of the $k$ best of $n$ candidates arriving in random order and the stopping rule is based on the relative ranks of the successive arrivals. For any sequence $(s_1, \ldots, s_k)$ describing the optimal policy, we also found the analytic formula for  the probability of success when using the  policy with this sequence. Our original approach is purely elementary and   bases on the combinatorial analysis of the problem. The obtained formulae reveal the possibility of an extension to an arbitrary value of $k$ for closed expressions describing the elements of an optimal sequence. Until now such  expressions were derived only for  $k\leq 3$. Since the maps $d_l$ in the inequalities (\ref{ineqqq}) describing the optimal sequence are all non-increasing, our formula reduces the determination of elements in this sequence to  solving a system of $k$ equations. In other words, we need to solve a  recurrence with the number of steps  bounded by $k$, which is substantially more advantageous than the implicit solution via computing the optimum from the known mechanism of dynamic or linear programming. On the other hands, in recent years, the linear programming approach was discovered  to analyze a broader class of secretary problems. For example, in \cite{7} the authors consider a so-called  $J$-choice $K$-best secretary problem (the case $J=1$ was the subject of the present paper), where finding of an optimal solution reduces to solving the corresponding linear program. In~\cite{10} the authors use linear programming but to the so-called continuous and infinite models of the secretary problem (see also \cite{8,9}). In~\cite{17} even a more general problem is studied via this technique -- a so called shared $Q$-queue $J$-choice $K$-best secretary problem. Therefore, it seems natural to analyze and develop our combinatorial approach also for wider classes of secretary problems, which  might result in finding some simplifications in the corresponding  formulae. The construction of the optimal sequence  from Theorem~\ref{mt} could also be  applied in the study of the limits $\tau_l(k):=\lim_{n\to\infty}s_l/n$ ($1\leq l\leq k$) and their behaviour. This could help in solving some (according to our knowledge) open questions concerning these limits, such as (see also \cite{3,13}): Is it true that $\tau_1(k)$  monotonically decreases with $k$?

\end{document}